\documentclass[reqno]{amsart}

\pdfoutput=1
\usepackage[utf8]{inputenc} 
\usepackage{bbm} 
\usepackage{commath} 
\usepackage{mathtools} 
\usepackage{todonotes} 
\usepackage[colorlinks, citecolor=red, breaklinks=true]{hyperref} 
\usepackage[stretch=10, shrink=10]{microtype} 
\usepackage[initials,msc-links,non-compressed-cites,nobysame]{amsrefs}[2007/10/22]
\usepackage{subcaption}

\newtheorem{theorem}{Theorem}[section]
\newtheorem{lemma}[theorem]{Lemma}
\newtheorem{proposition}[theorem]{Proposition}
\newtheorem{fact}[theorem]{Fact}

\theoremstyle{definition}
\newtheorem{definition}[theorem]{Definition}

\theoremstyle{remark}
\newtheorem{remark}[theorem]{Remark}
\newtheorem{example}[theorem]{Example}

\numberwithin{equation}{section}

\let\Im\relax
\DeclareMathOperator{\Im}{Im}
\let\Re\relax
\DeclareMathOperator{\Re}{Re}

\usepackage{paralist,enumitem}
\setlist{listparindent=0pt,parsep=3pt}

\newcommand{\TitleWithUrl}[1]{\IfEmptyBibField{doi}%
  {\IfEmptyBibField{url}{\textit{#1}}%
    {\IfEmptyBibField{eprint}{\href {\BibField{url}}{\textit{#1}}}{\textit{#1}}}%
    }%
  {\href {https://doi.org/\BibField{doi}}{\textit{#1}}}}
\renewcommand{\eprint}[1]{\IfEmptyBibField{url}{\url{#1}}%
  {\href {\BibField{url}}{#1}}}

\BibSpec{article}{%
    +{}  {\PrintAuthors}                {author}
    +{,} { \TitleWithUrl}               {title}
    +{.} { }                            {part}
    +{:} { \textit}                     {subtitle}
    +{,} { \PrintContributions}         {contribution}
    +{.} { \PrintPartials}              {partial}
    +{,} { }                            {journal}
    +{}  { \textbf}                     {volume}
    +{}  { \PrintDatePV}                {date}
    +{,} { \issuetext}                  {number}
    +{,} { \eprintpages}                {pages}
    +{,} { }                            {status}
    +{,} { available at arXiv:\eprint}        {eprint}
    +{}  { \parenthesize}               {language}
    +{}  { \PrintTranslation}           {translation}
    +{;} { \PrintReprint}               {reprint}
    +{.} { }                            {note}
    +{.} {}                             {transition}
    +{}  {\SentenceSpace \PrintReviews} {review}
}

\BibSpec{collection.article}{%
    +{}  {\PrintAuthors}                {author}
    +{,} { \TitleWithUrl}                     {title}
    +{.} { }                            {part}
    +{:} { \textit}                     {subtitle}
    +{,} { \PrintContributions}         {contribution}
    +{,} { \PrintConference}            {conference}
    +{}  {\PrintBook}                   {book}
    +{,} { }                            {booktitle}
    +{,} { \PrintDateB}                 {date}
    +{,} { pp.~}                        {pages}
    +{,} { }                            {status}
    +{,} { available at arXiv:\eprint}        {eprint}
    +{}  { \parenthesize}               {language}
    +{}  { \PrintTranslation}           {translation}
    +{;} { \PrintReprint}               {reprint}
    +{.} { }                            {note}
    +{.} {}                             {transition}
    +{}  {\SentenceSpace \PrintReviews} {review}
}
\BibSpec{book}{%
    +{}  {\PrintPrimary}                {transition}
    +{,} { \TitleWithUrl}               {title}
    +{.} { }                            {part}
    +{:} { \textit}                     {subtitle}
    +{,} { \PrintEdition}               {edition}
    +{}  { \PrintEditorsB}              {editor}
    +{,} { \PrintTranslatorsC}          {translator}
    +{,} { \PrintContributions}         {contribution}
    +{,} { }                            {series}
    +{,} { \voltext}                    {volume}
    +{,} { }                            {publisher}
    +{,} { }                            {organization}
    +{,} { }                            {address}
    +{,} { \PrintDateB}                 {date}
    +{,} { }                            {status}
    +{}  { \parenthesize}               {language}
    +{}  { \PrintTranslation}           {translation}
    +{;} { \PrintReprint}               {reprint}
    +{.} { }                            {note}
    +{.} {}                             {transition}
    +{}  {\SentenceSpace \PrintReviews} {review}
}

\BibSpec{thesis}{%
    +{}  {\PrintAuthors}                {author}
    +{.} { \TitleWithUrl}               {title}
    +{:} { \textit}                     {subtitle}
    +{,} { \PrintThesisType}            {type}
    +{,} { }                            {organization}
    +{} { \PrintDate}                  {date}
    +{,} { }                            {address}
    +{,} { \eprint}                     {eprint}
    +{,} { }                            {status}
    +{}  { \parenthesize}               {language}
    +{}  { \PrintTranslation}           {translation}
    +{;} { \PrintReprint}               {reprint}
    +{.} { }                            {note}
    +{.} {}                             {transition}
    +{}  {\SentenceSpace \PrintReviews} {review}
}

\title[Weierstrass representations of discrete cmc surfaces in $\mathbb{I}^3$]{Weierstrass representations of discrete constant mean curvature surfaces in isotropic space}

\author{Joseph Cho}
\address[Joseph Cho]{Global Leadership School, Handong Global University, 558 Handong-ro Buk-gu, Pohang, Gyeongsangbuk-do 37554, Republic of Korea}
\email{jcho@handong.edu}

\author{Masaya Hara}
\address[Masaya Hara]{Department of Mathematics, Graduate School of Science, Kobe university, 1-1 Rokkodai-cho, Nada-ku, Kobe 657-8501, Japan}
\email{mhara@math.kobe-u.ac.jp}

\begin{document}

\begin{abstract}
	In this paper, we obtain Weierstrass representations for discrete constant mean curvature surfaces in isotropic $3$-space, and use this to construct examples with discrete closed-form parametrizations.
\end{abstract}

\maketitle

\section{Introduction}
Weierstrass representations have been fundemental in obtaining a deeper understanding of various classes of surfaces that admit them, as they provide a powerful method to analyze and construct surfaces of high differential geometric interest.
This is exemplified by the rich study of surface classes that admit the representations, including
\begin{itemize}
    \item minimal surfaces in Euclidean space \cite{weierstrass_UntersuchungenUberFlachen_1866},
    \item maximal surfaces and timelike minimal surfaces in Minkowski space \cite{kobayashi_MaximalSurfaces$3$dimensional_1983, konderak_WeierstrassRepresentationTheorem_2005},  and
    \item constant mean curvature $1$ (cmc-$1$) surfaces in hyperbolic space \cite{bryant_SurfacesMeanCurvature_1987}.
\end{itemize}
The representations themselves have also been revisited from various differential geometric viewpoints, such as spin geometry \cite{kusner_SpinorRepresentationMinimal_1995}, singularity theory \cite{umehara_MaximalSurfacesSingularities_2006}, or transformation theory \cite{umehara_ParametrizationWeierstrassFormulae_1992, hertrich-jeromin_MobiusGeometrySurfaces_2001, pember_WeierstrasstypeRepresentations_2020}.
Non-zero constant mean curvature (cmc) surfaces in Euclidean space or Minkowski space also admit Weierstrass-type representations \cite{dorfmeister_WeierstrassTypeRepresentation_1998, brander_HolomorphicRepresentationConstant_2010}; however, these methods are different from the aforementioned representations as they utilize the theory of loop groups.

Zero mean curvature surfaces in isotropic space also admit Weierstrass representations, as first found by Strubecker \cite{strubecker_DifferentialgeometrieIsotropenRaumes_1942a}, and these representations have been recently rediscovered from various geometric viewpoints, for example, \cite{dasilva_HolomorphicRepresentationMinimal_2021, seo_ZeroMeanCurvature_2021, sato_$d$minimalSurfacesThreedimensional_2021, cho_SpinorRepresentationIsotropic_2024}.
However, the case of isotropic space stands out from other space forms as the class of cmc surfaces in isotropic space also admit Weierstrass representations \cite{cho_SpinorRepresentationIsotropic_2024} that is similar to the zero mean curvature case: any cmc-$H$ surface $Y$ in isotropic space can locally be represented as
    \begin{equation}\label{eqn:wRepStart}
        Y = \Re \int (\bar{h} + g, 1, -\mathbbm{i}) \omega
    \end{equation}
where $h$ is a holomorphic function, $g$ a meromorphic function, and $\omega$ a holomorphic $1$-form such that $g^2 \omega$ is holomorphic while $\dif{h} = H\omega$.

With the growth of discrete differential geometry, a field that seeks to study discrete geometric structures with integrability as its central tenet, various works have found the discrete counterparts of Weierstrass representations, such as
\begin{itemize}
    \item discrete minimal surfaces in Euclidean space \cite{bobenko_DiscreteIsothermicSurfaces_1996},
    \item discrete maximal surfaces in Minkowski space \cite{yasumoto_DiscreteMaximalSurfaces_2015},
    \item discrete cmc-$1$ surfaces in hyperbolic space \cite{hertrich-jeromin_TransformationsDiscreteIsothermic_2000}, or
    \item discrete zero mean curvature surfaces in isotropic space \cite{pember_DiscreteWeierstrasstypeRepresentations_2023}.
\end{itemize}
The common key ingredient in recovering these discrete counterparts 
is the characterization of smooth Weierstrass representations in terms of  \emph{transformation theory} of isothermic surfaces, those surfaces that admit conformal curvature line coordinates away from umbilics.
Building on these characterizations, many aforementioned works on discrete Weierstrass representations use the theory of discrete isothermic surfaces \cite{bobenko_DiscreteIsothermicSurfaces_1996} and its transformation theory \cite{hertrich-jeromin_DiscreteVersionDarboux_1999, hertrich-jeromin_MobiusGeometrySurfaces_2001} to obtain the results.

In this paper, our aim is to examine the case of cmc surfaces in isotropic space, and obtain a discrete analogue of Weierstrass representations \eqref{eqn:wRepStart} for these surfaces.
However, the Weierstrass representation of smooth cmc surfaces lack a known characterization in terms of transformation theory.
Therefore, we choose an alternate characterization that is unique to cmc surfaces in isotropic space \cite{seo_ZeroMeanCurvature_2021}: namely, any cmc surface in isotropic space can be constructed from a zero mean curvature surface and a sphere.
By reinterpreting the Weierstrass representation \eqref{eqn:wRepStart} in terms of this characterization, we will obtain a discrete analogue of the Weierstrass representation, and check that the resulting discrete surface indeed has constant mean curvature given in terms of \emph{mixed areas} \cite{pottmann_GeometryMultilayerFreeform_2007, pember_DiscreteWeierstrasstypeRepresentations_2023}.

To achieve this goal, we structure our paper as follows.

Section~\ref{sect:prem} is preparatory in nature, introducing smooth and discrete surface theory in isotropic $3$-space.
In detail, after introducing isotropic space in Section~\ref{sub:one}, Sections~\ref{sub:two} and \ref{sub:three} concern smooth surface theory, culminating in the important reinterpretation of Weierstrass representation for cmc surfaces in Example~\ref{exam:Wcmc}.
Then we switch our attention to discrete surface theory in Sections~\ref{sub:four} and \ref{sub:five}, reviewing the concept of circular nets and discrete isothermic surfaces in the context of isotropic space \cite{pottmann_DiscreteSurfacesIsotropic_2007}.
In doing so, we also closely examine the concept of \emph{discrete lightlike Gauss maps} of circular nets, as it allows us to consider mean curvatures for discrete surfaces via mixed areas.

Our main result is contained in Section~\ref{sect:main}: we will first use the characterization of cmc surfaces in terms of a zero mean curvature surface and a sphere, and create a discrete isothermic surface with identical property in Section~\ref{sub:3one}.
Then in Section~\ref{sub:3two}, we will construct a discrete lightlike Gauss map for the discrete isothermic surface, and use this to check that the surface indeed has constant mean curvature in Section~\ref{sub:3three}, culminating in the desired Weierstrass representation for discrete cmc surfaces in Theorem~\ref{thm:wRep}.
After discussing the existence of parallel discrete cmc surfaces in Theorem~\ref{thm:parcmc}, we put our representation to test by constructing concrete examples in Section~\ref{sub:3four}, including discrete Delaunay-type surfaces in isotropic space.
In particular, we obtain \emph{explicit closed-form parametrizations} of all our examples, a relatively rare result in discrete differential geometry.

\textbf{Acknowledgements.}
The first author gratefully acknowledges the support from Handong Global University via New Faculty Research Grant.
The second author gratefully acknowledges the support from JST SPRING, Grant Number JPMJSP2148, including the opportunity to visit TU Wien.

\section{Preliminaries} \label{sect:prem}
We first introduce isotropic $3$-space and the surface theory within, both smooth and discrete.

\subsection{Laguerre geometric model of isotropic \texorpdfstring{$3$}{3}-space}\label{sub:one}
In this paper, we will use the Laguerre geometric model of isotropic $3$-space, which we briefly review here.
(For a more detailed account, we refer the readers to \cite[Section~2.2]{cho_SpinorRepresentationIsotropic_2024}.)

Let $\mathbb{R}^{3,1}$ denote the Minkowski $4$-space equipped with the bilinear form $( \cdot , \cdot )$ of signature $(- + + +)$, and for any $\tilde{S} \in \mathbb{R}^{3,1}$, let
    \[
        \mathcal{L}_{\tilde{S}} := \{ X \in \mathbb{R}^{3,1} : (X - \tilde{S}, X - \tilde{S}) = 0\}
    \]
be the affine lightcone of $\mathbb{R}^{3,1}$ centered around $\tilde{S}$.
For brevity, we will write the lightcone $\mathcal{L}_{0}$ of $\mathbb{R}^{3,1}$ as $\mathcal{L}$.
Choosing $\mathfrak{p}, \tilde{\mathfrak{p}} \in \mathcal{L}$ such that $(\mathfrak{p}, \tilde{\mathfrak{p}}) = 1$, we have
    \[
        \langle \mathfrak{p}, \tilde{\mathfrak{p}} \rangle^\perp \cong \mathbb{R}^2,
    \]
allowing us to split $\mathbb{R}^{3,1}$ as
    \[
        \mathbb{R}^{3,1} = \mathbb{R}^2 \oplus_\perp (\langle \mathfrak{p} \rangle \oplus \langle \tilde{\mathfrak{p}} \rangle).
    \]
Thus, we may define projections $\pi_\mathfrak{p}: \mathbb{R}^{3,1} \to \langle \mathfrak{p} \rangle$ and $\pi
    _{\tilde{\mathfrak{p}}}: \mathbb{R}^{3,1} \to \langle \tilde{\mathfrak{p}} \rangle$ via
    \[
        \pi_\mathfrak{p}(X) := (X,\tilde{\mathfrak{p}}) \mathfrak{p} \quad\text{and}\quad \pi_{\tilde{\mathfrak{p}}}(X) := (X,\mathfrak{p}) \tilde{\mathfrak{p}},
    \]
which in turn allows us to define orthoprojection $\pi_2 : \mathbb{R}^{3,1} \to \mathbb{R}^2$ via
    \[
        \pi_2(X) := X - \pi_\mathfrak{p}(X) - \pi_{\tilde{\mathfrak{p}}}(X).
    \]

Under the splitting, isotropic $3$-space $\mathbb{I}^3$ is given by the kernel of $\pi_{\tilde{\mathfrak{p}}}$, that is, it is the subspace
    \[
        \mathbb{I}^3 = \{ X \in \mathbb{R}^{3,1} : (X, \mathfrak{p}) = 0\} = \mathbb{R}^2 \oplus \langle \mathfrak{p} \rangle.
    \]
Therefore, any $X \in \mathbb{I}^3$ can be written as
    \[
        X = \pi_2(X) + \pi_\mathfrak{p}(X) =: x + \chi \mathfrak{p}
    \]
for some $\chi \in \mathbb{R}$.
To make use of explicit coordinates for isotropic space, we normalize
    \[
        \mathfrak{p} = (1,0,0,1)^t \quad\text{and}\quad \tilde{\mathfrak{p}} = \frac{1}{2}(-1,0,0,1)^t
    \]
so that $\mathbb{I}^3$ admits a coordinate chart $\psi : \{(\mathbf{l},\mathbf{x},\mathbf{y})\} \to \mathbb{I}^3$ via
    \[
        \psi((\mathbf{l},\mathbf{x},\mathbf{y})) = (\mathbf{l},\mathbf{x},\mathbf{y},\mathbf{l})^t.
    \]
Then the induced metric on isotropic space reads
    \begin{equation}\label{eqn:metric}
        \dif{\mathbf{x}}^2 + \dif{\mathbf{y}}^2.
    \end{equation}
We will refer to both $\{(\mathbf{l},\mathbf{x},\mathbf{y},\mathbf{l})^t\}$ and $\{(\mathbf{l},\mathbf{x},\mathbf{y})\}$ as isotropic $3$-space, and call $\mathbf{l}$--direction ``vertical''.

Then \emph{(spacelike) spheres} $S$ in isotropic $3$-space are given by the affine lightcones via
    \[
        S = \mathbb{I}^3 \cap \mathcal{L}_{\tilde{S}}.
    \]
\emph{Elliptic circles} in isotropic $3$-space are then given by the intersection between spheres and spacelike $2$-planes, so that it is in the shape of a Euclidean ellipse (such that its orthoprojection to $\mathbb{R}^2$ is a circle), while \emph{parabolic circles} are given by the intersection between spheres and vertical $2$-planes, so that it is in the shape of a Euclidean parabola (with vertical axis).

\subsection{Smooth surface theory in isotropic space}\label{sub:two}
For some simply-connected domain $D \subset \mathbb{R}^2$, let $X : D \to \mathbb{I}^3$ be a spacelike immersion.
Viewing $X$ as a codimension two surface in $\mathbb{R}^{3,1}$, every fiber of the normal bundle is a vector space of induced signature $(1,1)$ having $\mathfrak{p}$ as a constant null section.
Thus, we may find a unique null section  $N : D \to \mathcal{L}$ of the normal bundle such that
    \[
        (\dif{X}, N) = 0 \quad\text{and}\quad (N, \mathfrak{p}) = 1.
    \]
This is the \emph{lightlike Gauss map} of $X$.

Under this setting, the first and second fundamental forms of $X$ are given by
    \[
        \dif{s}^2 = (\dif{X}, \dif{X}) \quad\text{and}\quad \mathrm{II} = -(\dif{X}, \dif{N}),
    \]
and the definitions of \emph{(extrinsic) Gauss curvature $K$} and \emph{mean curvature $H$} follow.
In particular, the mean curvature $H$ of a surface $X$ can be given explicitly in terms of the vertical coordinate as follows (see, for example, \cite[Equation~(5.2)]{strubecker_DifferentialgeometrieIsotropenRaumes_1942a}, \cite[Equation~(9)]{pottmann_DiscreteSurfacesIsotropic_2007}, or \cite[Lemma~3.4]{cho_SpinorRepresentationIsotropic_2024}):
\begin{fact}
    Let $X : D \to \mathbb{I}^3$ be a spacelike immersion parametrized by coordinates $(u,v) \in D$ via
        \[
            X(u,v) = (\mathbf{l}(u,v), \mathbf{x}(u,v), \mathbf{y}(u,v)).
        \]
    Then the mean curvature $H$ of $X$ is given by
        \[
            H = \frac{1}{2} \Delta_X \mathbf{l}
        \]
    where $\Delta_X$ denotes the Laplacian with respect to the induced metric on $X$.
\end{fact}
\begin{remark}
    We note here that since the induced metric of $X$ and $x := \pi_2 (X)$ are identical by \eqref{eqn:metric}, the mean curvature $H$ of $X$ can also be found via
        \[
            H = \frac{1}{2} \Delta_x \mathbf{l}.
        \]
\end{remark}
\begin{example}\label{exam:sphmean}
    As an example, let us find the mean curvature of a sphere.
    Let $S \in \mathbb{R}^{3,1}$ be given by
        \[
            \tilde{S} = c + r \tilde{\mathfrak{p}} = (c_0 - \tfrac{r}{2}, c_1, c_2, c_0 + \tfrac{r}{2})^t
        \]
    where $c = \pi_2(S) + \pi_\mathfrak{p}(S)$ and $r \tilde{\mathfrak{p}} = \pi_{\tilde{\mathfrak{p}}}(S)$, so that $S = \mathcal{L}_{\tilde{S}} \cap \mathbb{I}^3$ is a sphere in isotropic space with center $c \in \mathbb{I}^3$ and radius $r \in \mathbb{R}$.
    Then one can calculate directly that $X = (\mathbf{l},\mathbf{x},\mathbf{y},\mathbf{l})^t \in S$ if and only if
        \[
            \mathbf{l} = \frac{1}{2r}((\mathbf{x} - c_1)^2 + (\mathbf{y} - c_2)^2) + c_0.
        \]
    Therefore, we can view sphere $\tilde{S}$ with radius $r$ as the graph of a function
        \[
            f(\mathbf{x}, \mathbf{y}) = \frac{1}{2r}((\mathbf{x} - c_1)^2 + (\mathbf{y} - c_2)^2) + c_0,
        \]
    and calculate that the mean curvature $H$ of the sphere must satisfy
        \[
            H = \frac{1}{2}(f_{\mathbf{x}\mathbf{x}} + f_{\mathbf{y}\mathbf{y}}) = \frac{1}{r}.
        \]
\end{example}

\begin{example}[Weierstrass representations of zero mean curvature surfaces]\label{exam:Wzmc}
    If the mean curvature $H$ of a surface $X$ identically vanishes, then the surface is called a \emph{zero mean curvature surface}, and must locally be the graph of a harmonic function.
    Thus, zero mean curvature surfaces in isotropic $3$-space also admit Weierstrass representation \cite[Equation~(8.31)]{strubecker_DifferentialgeometrieIsotropenRaumes_1942a}: any zero mean curvature surface $X$ can locally be represented as
        \[
            X = \Re \int (g, 1, -\mathbbm{i}) \omega
        \]
    where $g$ is a meromorphic function and $\omega$ is a holomorphic $1$-form such that $g^2 \omega$ is holomorphic.
    Here, $\mathbbm{i}$ denotes the complex imaginary unit.
\end{example}

\subsection{Sum of surfaces as graphs}\label{sub:three}
For $\iota = 1, 2$, let $X_\iota : D \to \mathbb{I}^3$ be spacelike immersions so that $\pi_2(X_1) = \pi_2(X_2) = x$, while $\pi_\mathfrak{p}(X_\iota) =: \mathbf{l}_\iota \mathfrak{p}$ so that
    \[
        X_1 = x + \mathbf{l}_1 \mathfrak{p} \quad\text{and}\quad X_2 = x + \mathbf{l}_2 \mathfrak{p}.
    \]
Then the respective mean curvatures satisfy
    \[
        H_\iota = \frac{1}{2}\Delta_x \mathbf{l}_\iota.
    \]
Now let $X : D \to \mathbb{I}^3$ be the surface obtained by adding $X_1$ and $X_2$ as graphs, denoted by
    \[
        X = X_1 +_{\mathrm{gr}} X_2,
    \]
that is,
    \[
        X = x + (\mathbf{l}_1 + \mathbf{l}_2) \mathfrak{p} =: x + \mathbf{l} \mathfrak{p}.
    \]
Then the mean curvature $H$ of $X$ must be
    \begin{equation}\label{eqn:hgraphsum}
        H = \frac{1}{2} \Delta_x \mathbf{l} = \frac{1}{2} \Delta_x (\mathbf{l}_1 + \mathbf{l}_2) = H_1 + H_2.
    \end{equation}
Thus if a surface $X$ is obtained as a sum of two surfaces $X_1$ and $X_2$ viewed as graphs of functions, then the mean curvature $H$ of the original surface $X$ must be obtained by adding the mean curvatures $H_1$ and $H_2$ of the surfaces $X_1$ and $X_2$, respectively.

\begin{example}[Weierstrass representation of non-zero constant mean curvature (cmc) surface]\label{exam:Wcmc}
    In \cite[Theorem~3.10]{cho_SpinorRepresentationIsotropic_2024}, a Weierstrass representation for cmc-$H$ surfaces were given: any cmc-$H$ surface $Y$ can locally be obtained via
        \[
            Y = \Re \int (\bar{h} + g, 1, -\mathbbm{i}) \omega
        \]
    for some holomorphic $h$, meromorphic $g$, and holomorphic $1$-form $\omega$ such that
    \begin{enumerate}
        \item $g^2 \omega$ is holomorphic, and
        \item $\dif{h} = H\omega$.
    \end{enumerate}

    Let
        \[
            S := \Re \int (\bar{h}, 1, -\mathbbm{i}) \omega \quad\text{and}\quad X := \Re \int (g, 1, -\mathbbm{i}) \omega
        \]
    so that
        \[
            Y = S +_{\mathrm{gr}} X.
        \]
    Then by Example~\ref{exam:Wzmc} we see that $X$ is a zero mean curvature surface, while \cite[\S~3.5]{cho_SpinorRepresentationIsotropic_2024} tells us that $S$ is a sphere with radius $\tfrac{1}{H}$, so that it has mean curvature $H$ by Example~\ref{exam:sphmean}.
    Thus \eqref{eqn:hgraphsum} tells us that $Y$ indeed is a cmc-$H$ surface.
    
    In summary, Weierstrass representation for cmc-$H$ surfaces tells us that any cmc-$H$ surface $Y$ can locally be obtained as
        \begin{equation}\label{eqn:cmcChar}
            Y = S +_{\mathrm{gr}} X
        \end{equation}
    where $S$ is a sphere with mean curvature $H$, while $X$ is a zero mean curvature surface. (See also \cite[Section~7]{seo_ZeroMeanCurvature_2021}.)

    We shall use this characterization of cmc-$H$ surfaces in isotropic space to obtain a discrete analogue of the Weierstrass representation.
\end{example}

\subsection{Circular nets}\label{sub:four}
Now we shift our focus to discrete surfaces in isotropic $3$-space.
Let us assume throughout that $\Sigma \subset \mathbb{Z}^2$ is a discrete simply-connected domain in the sense of \cite[Definition~2.15]{burstall_Discrete$Omega$netsGuichard_2023} with $(m,n) \in \Sigma$.
We will often refer to an edge $((m,n),(m+1,n))$ or $((m,n),(m,n+1))$ as edge $(ij)$, or a face $((m,n),(m+1,n),(m+1,n+1),(m,n+1))$ as an \emph{elementary quadrilateral $(ijk\ell)$}.

For a discrete function $\mathfrak{f}$ defined on $\Sigma$, we will often write for any $i \in \Sigma$,
    \[
        \mathfrak{f}_i := \mathfrak{f}(i).
    \]
Then the discrete $1$-form $\dif{\mathfrak{f}}$ is given via 
    \[
        \dif{\mathfrak{f}}_{ij} := \mathfrak{f}_i - \mathfrak{f}_j
    \]
on any oriented edge $(ij)$.
We can then use
    \[
        \mathfrak{f}_{ij} := \frac{1}{2}(\mathfrak{f}_i + \mathfrak{f}_j),
    \]
to state the Leibniz rule for discrete functions:
\begin{fact}[{\cite[Equation~(2.2)]{burstall_DiscreteSurfacesConstant_2014}}]
    Let $\mathfrak{f}, \mathfrak{g}: \Sigma \to V$ where $V$ is a vector space equipped with a symmetric bilinear product $\odot$.
    Then we have
        \[
            \dif{(\mathfrak{f} \odot \mathfrak{g})}_{ij} = \dif{\mathfrak{f}}_{ij} \odot \mathfrak{g}_{ij} + \mathfrak{f}_{ij} \odot \dif{\mathfrak{g}}_{ij}.
        \]
\end{fact}

We will focus on the particular case of \emph{circular nets}:
\begin{definition}[{\cite[Theorem~1]{pottmann_DiscreteSurfacesIsotropic_2007}}]\label{def:circular}
    A discrete surface $X : \Sigma \to \mathbb{I}^3$ is called a \emph{circular net}, if on any elementary quadrilateral $(ijk\ell)$, we have $X_i$, $X_j$, $X_k$, and $X_\ell$ lies on an elliptic circle.
    Equivalently, $X$ satisfies the conditions
        \begin{enumerate}
            \item $X_i$, $X_j$, $X_k$, and $X_\ell$ are coplanar, and
            \item for $x = \pi_2(X) : \Sigma \to \mathbb{R}^2$, $x_i$, $x_j$, $x_k$, and $x_\ell$ are concircular
        \end{enumerate}
    on any elementary quadrilateral $(ijk\ell)$.
\end{definition}

To see how curvatures are defined on circular nets, we will discretize the notion of lightlike Gauss maps, and introduce
\emph{discrete lightlike Gauss maps}.
For this, let $X : \Sigma \to \mathbb{I}^3$ be a circular net, and define
    \[
        P := \{X \in \mathbb{R}^{3,1} : (X,\mathfrak{p}) = 1\}.
    \]
Now, let $N : \Sigma \to \mathcal{L} \cap P$ be determined by the condition
    \begin{equation}\label{eqn:dg}
        \dif{X}_{ij} \parallel \dif{N}_{ij} \neq 0
    \end{equation}
on any edge $(ij)$ generically.

We must first check that \eqref{eqn:dg} gives rise to a well-defined $N$, only depending on the choice of $N$ at a single vertex.

\begin{lemma}[cf.\ {\cite[Lemma~4.1.2]{cho_DiscreteIsothermicSurfaces_}}]
    The propogation of $N$ via \eqref{eqn:dg} is consistent.
\end{lemma}
\begin{proof}
To see this, first fix any edge $(ij)$, and let $\alpha$ be some real-valued positive function defined on unoriented edges so that
    \[
        N_i - N_j = \dif{N}_{ij} = -\alpha_{ij} \dif{X}_{ij}.
    \]
To find $\alpha_{ij}$, we use
    \[
        N_j = \alpha_{ij} \dif{X}_{ij} + N_i,
    \]
to calculate
    \[
        0 = (N_j, N_j) = \alpha_{ij}^2 (\dif{X}_{ij}, \dif{X}_{ij}) + 2 \alpha_{ij} (\dif{X}_{ij}, N_i),
    \]
and conclude that
    \[
        \alpha_{ij} = -\frac{2 (\dif{X}_{ij}, N_i)}{(\dif{X}_{ij}, \dif{X}_{ij})}
    \]
since $\alpha_{ij} \neq 0$.
Therefore, on any edge $(ij)$, $N_j$ is uniquely determined from $\dif{X}_{ij}$ and $N_i$ via
    \begin{equation}\label{eqn:gj}
        N_j =  -\frac{2 (\dif{X}_{ij}, N_i)}{(\dif{X}_{ij}, \dif{X}_{ij})} \dif{X}_{ij} + N_i.
    \end{equation}

Now, by applying appropriate rotations and translations, assume without loss of generality that $X_i$, $X_j$, $X_k$, and $X_\ell$ lie on a unit circle in $\mathbb{R}^2 \cong \langle \mathfrak{p},\tilde{\mathfrak{p}}\rangle^\perp$, that is, there is some $\theta_\iota \in \mathbb{R}$ for $\iota \in \{i,j,k,\ell\}$ such that,
    \[
        X_\iota = (0, \cos \theta_\iota, \sin \theta_\iota, 0)^t.
    \]
Defining
    \begin{align*}
        \mathcal{T}_\iota &:= (0, -\sin \theta_\iota, \cos \theta_\iota, 0)^t \\
        \mathcal{N}_\iota &:= (0, -\cos\theta_\iota, -\sin \theta_\iota, 0)^t,
    \end{align*}
and taking $\{\mathcal{T}_\iota, \mathcal{N}_\iota, \mathfrak{p}, \tilde{\mathfrak{p}}\}$ as a basis of $\mathbb{R}^{3,1}$ at $X_\iota$, we see that $N_\iota \in \mathcal{L} \cap P$ implies 
    \begin{equation}\label{eqn:gForm}
        N_\iota = A_\iota \mathcal{T}_\iota + B_\iota \mathcal{N}_\iota - \tfrac{1}{2}(A_\iota^2 + B_\iota^2) \mathfrak{p} +  \tilde{\mathfrak{p}}
    \end{equation}
for some $A_\iota, B_\iota \in \mathbb{R}$.
Under this setting, we use \eqref{eqn:gj} to calculate $N_j$ from $N_i$ and find that
    \begin{equation}\label{eqn:gProp}
        N_j = -A_i \mathcal{T}_j + B_i \mathcal{N}_j - \tfrac{1}{2}(A_i^2 + B_i^2)\mathfrak{p} + \tilde{\mathfrak{p}}.
    \end{equation}

Thus, for any initial choice of $N$ at a single vertex, say $i \in \Sigma$, we can now calculate $N$ around an elementary quadrilateral $(ijk\ell)$ using the propogation formula \eqref{eqn:gProp} to see that
    \begin{align*}
        N_i &= A_i \mathcal{T}_i + B_i \mathcal{N}_i - \tfrac{1}{2}(A_i^2 + B_i^2) \mathfrak{p} + \tilde{\mathfrak{p}} \\
        N_j &= -A_i \mathcal{T}_j + B_i \mathcal{N}_j - \tfrac{1}{2}(A_i^2 + B_i^2) \mathfrak{p} + \tilde{\mathfrak{p}} \\
        N_k &= A_i \mathcal{T}_k + B_i \mathcal{N}_k - \tfrac{1}{2}(A_i^2 + B_i^2) \mathfrak{p} + \tilde{\mathfrak{p}} \\
        N_\ell &= -A_i \mathcal{T}_\ell + B_i \mathcal{N}_\ell - \tfrac{1}{2}(A_i^2 + B_i^2) \mathfrak{p} + \tilde{\mathfrak{p}}.
    \end{align*}
Finally, denoting the propogation of $N$ along the edge $(\ell i)$ from $N_\ell$ as $\tilde{N}_i$, we find that
    \[
        \tilde{N}_i = A_i \mathcal{T}_i + B_i \mathcal{N}_i - \tfrac{1}{2}(A_i^2 + B_i^2) \mathfrak{p} + \tilde{\mathfrak{p}} = N_i
    \]
so that the propogation is consistent.
\end{proof}

\begin{remark}
    Since $X$ and $N$ are edge-parallel, or \emph{parallel nets}, $N$ is also a circular net (see \cite[p.\ 156]{bobenko_DiscreteDifferentialGeometry_2008a} or \cite[Proposition~4.10]{burstall_Discrete$Omega$netsGuichard_2023}).
\end{remark}

\begin{remark}
Note that the expression for $N$ as in \eqref{eqn:gForm} shows that
    \[
        \nu := N - \tilde{\mathfrak{p}}
    \]
not only takes values in $\mathbb{I}^3$, but also it takes values in the unit sphere of isotropic space (see, for example, \cite[Equation~(7.1)]{strubecker_DifferentialgeometrieIsotropenRaumes_1942a}):
    \[
        \mathbf{l} = -\frac{1}{2}(\mathbf{x}^2 + \mathbf{y}^2).
    \]
Such $\nu$ is called the \emph{Gauss map} of $X$, and we have
    \[
        \dif{\nu}_{ij} = \dif{N}_{ij} \parallel \dif{X}_{ij}.
    \]
\end{remark}

Now that lightlike Gauss map $N$ of a circular net $X$ is well-defined, we may consider the definition of mean curvature using the notion of mixed areas:
\begin{definition}[{\cite[p.\ 822]{pember_DiscreteWeierstrasstypeRepresentations_2023} (cf.\ \cite[Lemma~2.12]{burstall_Discrete$Omega$netsGuichard_2023})}]
    Let $X, \tilde{X}: \Sigma \to \mathbb{R}^{3,1}$ be edge-parallel circular nets.
    The $\wedge^2 \mathbb{R}^{3,1}$--valued \emph{mixed area} of $X$ and $\tilde{X}$, denoted by $A(X, \tilde{X})$, is given by
        \[
            A(X, \tilde{X})_{ijk\ell} := \frac{1}{4}((X_i - X_k) \wedge (\tilde{X}_j - \tilde{X}_\ell) - (X_j - X_\ell) \wedge (\tilde{X}_i - \tilde{X}_k)).
        \]
\end{definition}
\begin{remark}
    It follows directly that the notion of mixed area is bilinear on the space of pairwise edge-parallel circular nets.
\end{remark}

Using the notion mixed areas, mean curvature of a circular net $X$ is given via the lightlike Gauss map $N$:
\begin{definition}[{\cite[Definition~3.2]{pember_DiscreteWeierstrasstypeRepresentations_2023}}]\label{def:mean}
    Let $X: \Sigma \to \mathbb{I}^3$ be a circular net with lightlike Gauss map $N$.
    The mean curvature $H$ of $X$ over an elementary quadrilateral $(ijk\ell)$ is given by the relation
        \[
            A(X,N)_{ijk\ell} + H_{ijk\ell} A(X, X)_{ijk\ell} = 0.
        \]
\end{definition}
\begin{remark}
    Since for any diagonal $(ik)$ over an elementary quadrilateral $(ijk\ell)$, we have
        \[
            N_i - N_k = \nu_i - \nu_k
        \]
    where $\nu = N - \tilde{\mathfrak{p}}$ is the corresponding Gauss map, we can rewrite the mean curvature as in \cite[Examples~3.3(ii)]{pember_DiscreteWeierstrasstypeRepresentations_2023}:
        \begin{equation}\label{eqn:meanCurv1}
            A(X,\nu)_{ijk\ell} + H_{ijk\ell} A(X, X)_{ijk\ell} = 0.
        \end{equation}
    
    Now that the mixed areas in the expression for the mean curvature \eqref{eqn:meanCurv1} only involves circular nets in $\mathbb{I}^3$, we can use \eqref{eqn:metric} to further simplify the expression as in \cite[Equation~(17)]{pottmann_DiscreteSurfacesIsotropic_2007} in terms of $x := \pi_2(X)$ and $n := \pi_2(\nu)$:
        \begin{equation}\label{eqn:meanCurv2}
            A(x,n)_{ijk\ell} + H_{ijk\ell} A(x, x)_{ijk\ell} = 0.
        \end{equation}
\end{remark}

\subsection{Discrete isothermic surfaces}\label{sub:five}
Smooth isothermic surfaces are characterized as surfaces that admit conformal curvature line coordinates away from umbilics; as circular nets are discretizations of curvature line parametrized surfaces \cite{nutbourne_DifferentialGeometryApplied_1988}, discrete isothermic surfaces are defined as a subclass of circular nets using the notion of discrete holomorphic functions in $\mathbb{C}$:
\begin{definition}[{\cite[Definition~8]{bobenko_DiscreteIsothermicSurfaces_1996}}]\label{def:discHolo}
    A discrete map $x : \Sigma \to \mathbb{C} \cong \mathbb{R}^2$ is called a \emph{discrete holomorphic function} if on every elementary quadrilateral, we have
        \[
            \mathrm{cr}(x_i, x_j, x_k, x_\ell) := \frac{x_i- x_j}{x_j- x_k}\frac{x_k- x_\ell}{x_\ell- x_i} = \frac{m_{i\ell}}{m_{ij}}
        \]
    for some real non-vanishing discrete function $m$ defined on unoriented edges satisfying the \emph{edge-labeling} property, that is,
        \[
            m_{ij} = m_{\ell k} \quad\text{and}\quad m_{i\ell} = m_{jk}.
        \]
    The discrete function $m$ is called a \emph{cross-ratios factorizing function} of $x$, and is determined up to multiplication by a constant.
\end{definition}
\begin{definition}[{\cite[Theorem~3]{pottmann_DiscreteSurfacesIsotropic_2007}}]\label{def:isothermic}
    A circular net $X : \Sigma \to \mathbb{I}^3$ is called a \emph{discrete isothermic surface} if $x := \pi_2(X): \Sigma \to \mathbb{R}^2 \cong \mathbb{C}$ is a discrete holomorphic function.
\end{definition}

These notions can be combined to give the Weierstrass representation for discrete zero mean curvature surfaces:
\begin{fact}[{\cite[Theorem~4.1]{pember_DiscreteWeierstrasstypeRepresentations_2023} (cf.\ \cite[\S~4]{pottmann_DiscreteSurfacesIsotropic_2007})}]\label{fact:Wrepmin}
    Let $g: \Sigma \to \mathbb{C}$ be a discrete holomorphic function. 
    If $X: \Sigma \to \mathbb{I}^3$ is given recursively via
        \[
            \dif{X}_{ij} = \frac{1}{m_{ij}} \Re\left((g_{ij}, 1, -\mathbbm{i}) \frac{1}{\dif{g}_{ij}}\right),
        \]
    then $X$ is a discrete isothermic zero mean curvature surface.
\end{fact}

\section{Discrete cmc surfaces in isotropic space}\label{sect:main}
In this section, we will use the characterization of smooth cmc-$H$ surfaces in isotropic $3$-space obtained in Example~\ref{exam:Wcmc} to obtain a Weierstrass representation for discrete cmc surfaces in isotropic space, and show that these discrete surfaces indeed have constant mean curvature by choosing a suitable lightlike Gauss map.

\subsection{Adding discrete surfaces as graphs}
As the characterization of smooth cmc-$H$ surfaces obtained in  Example~\ref{exam:Wcmc} depend on the notion of adding two surfaces as graphs, we will first define a similar notion for discrete surfaces:
\begin{definition}
    Let $X_1, X_2: \Sigma \to \mathbb{I}^3$ be discrete surfaces such that
        \[
            \pi_2(X_1) = \pi_2(X_2) = x : \Sigma \to \mathbb{R}^2
        \]
    so that there exists functions $\mathbf{l}_1, \mathbf{l}_2: \Sigma \to \mathbb{R}$ with
        \[
            X_1 = x + \mathbf{l}_1 \mathfrak{p} \quad\text{and}\quad X_2 = x + \mathbf{l}_2 \mathfrak{p}.
        \]
    Then the discrete surface $X: \Sigma \to \mathbb{I}^3$ obtained via
        \[
            X := x + (\mathbf{l}_1 + \mathbf{l}_2) \mathfrak{p}
        \]
    is called a \emph{sum of $X_1$ and $X_2$ as graphs}, denoted as
        \[
            X = X_1 +_\mathrm{gr} X_2.
        \]
\end{definition}

The definition of circular nets in Definition~\ref{def:circular} and discrete isothermic surfaces in Definition~\ref{def:isothermic} immediately yields the following:
\begin{proposition}\label{prop:circAdd}
    Let $X_1, X_2: \Sigma \to \mathbb{I}^3$ be circular nets with $\pi_2(X_1) = \pi_2(X_2)$.
    If $X = X_1 +_\mathrm{gr} X_2$, then $X$ is also a circular net.
    Furthermore, if $X_1$ and $X_2$ are discrete isothermic surfaces, then $X$ is also a discrete isothermic surface.
\end{proposition}

\subsection{Construction of the discrete surface}\label{sub:3one}
We will construct a discrete surface $Y : \Sigma \to \mathbb{I}^3$ from a discrete isothermic zero mean curvature surface $X: \Sigma \to \mathbb{I}^3$ following the smooth characterization in \eqref{eqn:cmcChar}, that is, 
    \[
        Y := X +_\mathrm{gr} S
    \]
for some mapping $S$ into a sphere in isotropic space.

To start, let us fix some discrete isothermic zero mean curvature surface $X: \Sigma \to \mathbb{I}^3$ arising from a discrete holomorphic function $g: \Sigma \to \mathbb{C}$ with cross-ratios factorizing function $m$ via Weierstrass representation in Fact~\ref{fact:Wrepmin}, that is,
    \begin{equation}\label{eqn:dxij}
        \dif{X}_{ij} = \frac{1}{m_{ij}} \Re\left((g_{ij}, 1, -\mathbbm{i}) \frac{1}{\dif{g}_{ij}}\right).
    \end{equation}
Now let $S : \Sigma \to \mathbb{I}^3$ be a map into a sphere of non-zero mean curvature $H \neq 0$.
We may assume without loss of generality that $S$ is centered at the origin, so that
    \[
        S : \mathbf{l} = \frac{H}{2}(\mathbf{x}^2 + \mathbf{y}^2).
    \]
To add $X$ and $S$ as graphs we require that $\pi_2(S) = \pi_2(X) =: x$, allowing us to write
    \[
        S = x + \frac{H}{2} (x,x)\mathfrak{p}.
    \]
In fact, so obtained $S$ is a circular net:
\begin{lemma}\label{lemm:Scirc2}
    $S$ is a circular net.
\end{lemma}
\begin{proof}
    For any elementary quadrilateral $(ijk\ell)$, we will show that $S_i$, $S_j$, $S_k$, and $S_\ell$ are coplanar.
    As we know that $x_i$, $x_j$, $x_k$, and $x_\ell$ are concircular since $X$ is discrete isothermic, we can deduce that $S_i$, $S_j$, $S_k$, and $S_\ell$ must lie on the intersection of the sphere
        \begin{equation}\label{eqn:sphC1}
            \mathbf{l} = \tfrac{H}{2}(\mathbf{x}^2 + \mathbf{y}^2)
        \end{equation}
    and the circular cylinder
        \begin{equation}\label{eqn:sphC2}
            (\mathbf{x}-x_1)^2 + (\mathbf{y}-x_2)^2 = r^2
        \end{equation}
    for some real constants $x_1$, $x_2$, and $r > 0$.

    From \eqref{eqn:sphC2}, we get
        \[
            \mathbf{x}^2 + \mathbf{y}^2  = 2x_1 \mathbf{x} - x_1^2  + 2x_2 \mathbf{y} - x_2^2 + r^2,
        \]
    and substituting into \eqref{eqn:sphC1}, we have
        \[
            \mathbf{l} = H x_1 \mathbf{x} + H x_2 \mathbf{y} + \tfrac{H}{2}(r^2 - x_1^2  - x_2^2).
        \]
    Thus, the intersection of a sphere and a circular cylinder must lie on a plane.
\end{proof}

Now we wish to find an explicit parametrization of $S$ in terms of $g$.
To achieve this, first note that since $g$ is a discrete holomorphic function, it is a discrete isothermic surface; thus $g$ admits a Christoffel dual $h : \Sigma \to \mathbb{C}$ (see \cite[Theorem~6]{bobenko_DiscreteIsothermicSurfaces_1996} or \cite[Equation~(1.5)]{agafonov_AsymptoticBehaviorDiscrete_2005}), that is, $h$ is well-defined recursively via
    \begin{equation}\label{eqn:defh}
        \dif{h}_{ij} := \frac{H}{m_{ij}\dif{g}_{ij}}.
    \end{equation}
Such $h$ is also a discrete holomorphic function with cross-ratios factorizing function $\tilde{m} = \frac{1}{m}$.

\begin{remark}\label{rema:dual}
    Since $h$ is a Christoffel dual of $g$ in the complex plane, we have that $h$ is a Christoffel dual of $\bar{g}$ when viewed as maps into $\mathbb{R}^2$ (see \cite[\S~5.2.2]{hertrich-jeromin_IntroductionMobiusDifferential_2003}).
    Thus, by \cite[Equation~(8)]{pottmann_GeometryMultilayerFreeform_2007}, \cite[Theorem~4.42]{bobenko_DiscreteDifferentialGeometry_2008a}, or \cite[Theorem~4.11]{burstall_Discrete$Omega$netsGuichard_2023}, we have
        \[
            A(h, \bar{g}) \equiv 0.
        \]
\end{remark}
Using $h$, we can now rewrite $\dif{X}_{ij}$ in \eqref{eqn:dxij} as
    \begin{equation}\label{eqn:dXijdone}
        \dif{X}_{ij} = \frac{1}{m_{ij}} \Re\left((g_{ij}, 1, -\mathbbm{i}) \frac{1}{\dif{g}_{ij}}\right) = \frac{1}{H} \Re\left((g_{ij}, 1, -\mathbbm{i}) \dif{h}_{ij}\right),
    \end{equation}
which immediately yields
    \[
        \dif{x}_{ij} = \frac{1}{H}\Re\left((0, 1, -\mathbbm{i}) \dif{h}_{ij}\right).
    \]
Therefore, without loss of generality, we may assume that
    \begin{equation}\label{eqn:expx}
        x = \frac{1}{H}\Re((0,1,-\mathbbm{i})h) = \frac{1}{H}(0, \Re h, \Im h).
    \end{equation}

This allows us to write $S$ explicitly since
    \[
        (x,x) = \frac{1}{H^2}((\Re h)^2 + (\Im h)^2) = \frac{|h|^2}{H^2}
    \]
and thus
    \[
        S = x + \tfrac{H}{2} (x,x)\mathfrak{p} = \left(\tfrac{1}{2H} |h|^2, \tfrac{1}{H}\Re h, \tfrac{1}{H}\Im h\right) = \tfrac{1}{H}\Re \left(\left(\tfrac{1}{2} \bar{h}, 1, -\mathbbm{i}\right) h\right).
    \]

Using the fact that
    \begin{align}
        \Re(\bar{h}_{ij}\dif{h}_{ij})
            &= \frac{1}{2}\Re(|h_i|^2 - \bar{h}_i h_j + h_i \bar{h}_j - |h_j|^2) \nonumber\\
            &= \frac{1}{2}\Re(|h_i|^2 + 2 \mathbbm{i} \Im(h_i \bar{h}_j) - |h_j|^2)\nonumber\\
            &= \frac{1}{2}\Re(|h_i|^2 - |h_j|^2), \label{eqn:weird}
    \end{align}
we can verify
    \begin{align}
        \dif{S}_{ij}
            &= \tfrac{1}{H}\Re \left(\left(\tfrac{1}{2} \bar{h}_i, 1, -\mathbbm{i}\right) h_i\right) - \Re \left(\left(\tfrac{1}{2} \bar{h}_j, 1, -\mathbbm{i}\right) h_j\right) \nonumber\\
            &= \tfrac{1}{H}\Re \left(\left(\tfrac{1}{2} |h_i|^2, h_i, -\mathbbm{i} h_i\right) - \left(\tfrac{1}{2} |h_j|^2, h_j, -\mathbbm{i}h_j \right)\right) \nonumber\\
            &= \tfrac{1}{H}\Re \left(\bar{h}_{ij}\dif{h}_{ij}, \dif{h}_{ij}, -\mathbbm{i} \dif{h}_{ij}\right) \nonumber\\
            &= \tfrac{1}{H}\Re \left(\left(\bar{h}_{ij}, 1 , -\mathbbm{i} \right)\dif{h}_{ij}\right) \label{eqn:dSijdone}.
    \end{align}
Therefore, if $Y := X +_\mathrm{gr} S$, then \eqref{eqn:dXijdone} and \eqref{eqn:dSijdone} implies
    \[
        \dif{Y}_{ij} = \frac{1}{H} \Re\left((\bar{h}_{ij} + g_{ij}, 1, -\mathbbm{i}) \dif{h}_{ij}\right).
    \]
Finally, Proposition~\ref{prop:circAdd} and Lemma~\ref{lemm:Scirc2} allows us to conclude $Y$ is a discrete isothermic surface.
Summarizing:
\begin{proposition}
    For a discrete holomorphic function $g: \Sigma \to \mathbb{C}$ with cross-ratios factorizing function $m$, a discrete isothermic surface $Y : \Sigma \to \mathbb{I}^3$ can be defined recursively via
    \begin{equation}\label{eqn:cmcWrepT}
        \dif{Y}_{ij} = \frac{1}{H} \Re\left((\bar{h}_{ij} + g_{ij}, 1, -\mathbbm{i}) \dif{h}_{ij}\right)
    \end{equation}
    where
        \[
            \dif{h}_{ij} = \frac{H}{m_{ij}\dif{g}_{ij}}.
        \]
\end{proposition}

\subsection{Construction of the lightlike Gauss map}\label{sub:3two}
We will now construct the lightlike Gauss map $N : \Sigma \to \mathcal{L} \cap P$ of the discrete isothermic surface $Y : \Sigma \to \mathbb{I}^3$ obtained via \eqref{eqn:cmcWrepT}.

Defining $\varphi := \bar{h} + g$, if we take
    \[
        N := -\frac{1}{2}(1 + |\varphi|^2, 2 \Re \varphi, -2 \Im \varphi, -1 + |\varphi|^2)^t,
    \]
then we can directly verify that
    \[
        (N,N) = 0 \quad\text{and}\quad (N, \mathfrak{p}) = 1.
    \]
We will use $\nu = N - \tilde{\mathfrak{p}} : \Sigma \to \mathbb{I}^3$ given by
    \begin{equation}\label{eqn:expnu}
        \nu = -\frac{1}{2}(|\varphi|^2, 2\Re \varphi, -2 \Im \varphi)
    \end{equation}
to show the following:
\begin{lemma}\label{lemma:Nprop}
    For a real function $\beta$ defined on unoriented edges of $\Sigma$ given by
        \[
            \beta_{ij} := -m_{ij}|{\dif{g}_{ij}}|^2 - H,
        \]
    we have
        \[
            \dif{N}_{ij} = \dif{\nu}_{ij} = \beta_{ij} \dif{Y}_{ij}.
        \]
\end{lemma}
\begin{proof}
    Using a similar relation to \eqref{eqn:weird} for $\varphi$, we may write
        \begin{align}
            \dif{\nu}_{ij}
                &= \left(-\tfrac{1}{2}(|\varphi_i|^2 - |\varphi_j|^2), -\Re \dif{\varphi}_{ij}, \Im\dif{\varphi}_{ij}\right) \nonumber\\
                &= \left(-\Re(\varphi_{ij} \dif{\bar{\varphi}}_{ij}), -\Re \dif{\bar{\varphi}}_{ij}, \Re(\mathbbm{i}\dif{\bar{\varphi}}_{ij})\right) \nonumber\\
                &= \Re \left((-\varphi_{ij}, -1, \mathbbm{i})\dif{\bar{\varphi}}_{ij}\right) \label{eqn:dnu}.
        \end{align}
    Therefore, using the definition of $\dif{h}_{ij}$ \eqref{eqn:defh}
        \begin{align*}
            \dif{\nu}_{ij} - \beta_{ij} \dif{Y}_{ij}
                &= -\Re \left((\varphi_{ij}, 1, -\mathbbm{i})\left( \dif{h}_{ij} + \dif{\bar{g}_{ij}} + \tfrac{\beta_{ij}}{H} \dif{h}_{ij}\right)\right) \\
                &= -\Re \left((\varphi_{ij}, 1, -\mathbbm{i})\left( \tfrac{H + \beta_{ij}}{H}\tfrac{H}{m_{ij}\dif{g}_{ij}} + \dif{\bar{g}_{ij}}\right)\right) \\
                &= -\Re \left((\varphi_{ij}, 1, -\mathbbm{i}) (-{\dif{\bar{g}}_{ij}} + \dif{\bar{g}_{ij}})\right)\\
                &=0,
        \end{align*}
    giving us the desired conclusion.
\end{proof}

Lemma~\ref{lemma:Nprop} tells us that $N$ is a lightlike Gauss map of $Y$, allowing us to summarize as follows:
\begin{proposition}
    The discrete isothermic surface $Y: \Sigma \to \mathbb{I}^3$ given by
    \begin{equation}\label{eqn:cmcWrep1}
        \dif{Y}_{ij} = \frac{1}{H} \Re\left((\bar{h}_{ij} + g_{ij}, 1, -\mathbbm{i}) \dif{h}_{ij}\right)
    \end{equation}
    with
        \[
            \dif{h}_{ij} = \frac{H}{m_{ij}\dif{g}_{ij}}
        \]
    admits a lightlike Gauss map $N : \Sigma \to \mathcal{L} \cap P$ of the form
    \begin{equation}\label{eqn:lGaussMap}
        N = -\frac{1}{2}(1 + |\varphi|^2, 2 \Re \varphi, -2 \Im \varphi, -1 + |\varphi|^2)^t
    \end{equation}
    where
        \[
            \varphi := \bar{h} + g.
        \]
\end{proposition}

\subsection{Weierstrass representation for discrete cmc surface}\label{sub:3three}
Finally, we will use the lightlike Gauss map $N$ in \eqref{eqn:lGaussMap} to show that $Y$ obtained via \eqref{eqn:cmcWrep1} has constant mean curvature $H$.
In particular, we will make use of the expression \eqref{eqn:meanCurv2}: thus let us define $x := \pi_2(X)$ and $n := \pi_2(\nu)$.
Viewing the discrete complex functions $g$, $h$, $\varphi$, and their conjugates as maps into $\mathbb{R}^2$, the expression for $x$ \eqref{eqn:expx} tells us that
    \[
        x = \frac{1}{H}(0,\Re h, \Im h) = \frac{1}{H}h
    \]
while the expression for $\nu$ \eqref{eqn:expnu} allows us to see that
    \[
        n = -(0, \Re \varphi, - \Im \varphi) = -\bar{\varphi} = -h - \bar{g}.
    \]
Using these, we can calculate that
    \[
        A(x,n)_{ijk\ell} + H A(x, x)_{ijk\ell} = A(x, n + Hx) = A\left(\tfrac{1}{H}h,- \bar{g}\right) = -\tfrac{1}{H}A(h,\bar{g}),
    \]
which vanishes by Remark~\ref{rema:dual}.
Hence, we have $H_{ijk\ell} \equiv H$, that is, $Y$ has constant mean curvature $H$, allowing us to conclude that \eqref{eqn:cmcWrep1} gives a Weierstrass representation for discrete isothermic cmc surfaces:
\begin{theorem}[Weierstrass representation for discrete cmc surfaces]\label{thm:wRep}
    Let $g: \Sigma \to \mathbb{C}$ be a discrete holomorphic function with cross-ratios factorizing function $m$, and let $h$ be a discrete holomorphic function defined on any edge $(ij)$ via
        \begin{equation}\label{eqn:ghcond}
            \dif{h}_{ij} = \frac{H}{m_{ij}\dif{g}_{ij}} =: H\omega_{ij}
        \end{equation}
    for some real non-zero constant $H$.
    Then the discrete $1$-form
        \begin{equation}\label{eqn:cmcWrep2}
            \dif{Y}_{ij} = \Re\left((\bar{h}_{ij} + g_{ij}, 1, -\mathbbm{i}) \,\omega_{ij}\right)
        \end{equation}
    on any edge $(ij)$ defines a discrete isothermic constant mean curvature $H$ surface $Y: \Sigma \to \mathbb{I}^3$ with lightlike Gauss map $N : \Sigma \to \mathcal{L} \cap P$ given by
    \begin{equation}\label{eqn:lGaussMap2}
        N = -\frac{1}{2}(1 + |\varphi|^2, 2 \Re \varphi, -2 \Im \varphi, -1 + |\varphi|^2)^t
    \end{equation}
    for
        \[
            \varphi = \bar{h} + g.
        \]
    In this case, we say that $(h,g)$ are the Weierstrass data of $Y$.
\end{theorem}

As in the smooth case, discrete cmc surfaces in Euclidean space admit a parallel cmc surface \cite[p.\ 23]{hertrich-jeromin_DiscreteVersionDarboux_1999}.
We will now show the isotropic analogue of such result using the Weierstrass representation:

\begin{theorem}\label{thm:parcmc}
    Let $Y: \Sigma \to \mathbb{I}^3$ be a discrete isothermic cmc-$H$ surface with Gauss map $\nu$ given by the Weierstrass data $(h,g)$.
    Then $Y^P: \Sigma \to \mathbb{I}^3$ defined by
        \begin{equation}\label{eqn:defpar}
            Y^P := Y + \frac{1}{H} \nu
        \end{equation}
    is a parallel cmc-$(-H)$ surface with Weierstrass data
        \begin{equation}\label{eqn:wdataPar}
            (h^P, g^P) = (\bar{g}, \bar{h}).
        \end{equation}
\end{theorem}
\begin{proof}
    From the definition of $Y^P$ \eqref{eqn:defpar}, it is easy to see that $Y^P$ is edge-parallel to $Y$.
    Let us define discrete holomorphic functions $g^P$ and $h^P$ via \eqref{eqn:wdataPar}.
    Now we calculate using \eqref{eqn:cmcWrep1} and \eqref{eqn:dnu} that
        \begin{align*}
            \dif{Y}^P_{ij} &= \dif{Y}_{ij} + \frac{1}{H}\dif{\nu}_{ij} \\
                &= \frac{1}{H} \Re\left((\varphi_{ij}, 1, -\mathbbm{i}) \dif{h}_{ij}\right) - \frac{1}{H} \Re \left((\varphi_{ij}, 1, -\mathbbm{i})\dif{\bar{\varphi}}_{ij}\right) \\
                &= \frac{1}{H} \Re\left((\varphi_{ij}, 1, -\mathbbm{i}) (\dif{h}_{ij} - \dif{\bar{\varphi}}_{ij})\right)\\
                &= -\frac{1}{H} \Re\left((g_{ij} + \bar{h}_{ij}, 1, -\mathbbm{i}) \dif{\bar{g}}_{ij})\right)\\
                &= -\frac{1}{H} \Re\left((\bar{h}^P_{ij} + g^P_{ij}, 1, -\mathbbm{i}) \dif{h}^P_{ij})\right).
        \end{align*}
    Therefore, the Weierstrass representation for discrete cmc surfaces allows us to conclude that $Y^P$ is a cmc-$(-H)$ surface with Weierstrass data $(\bar{g}, \bar{h})$.
\end{proof}

\subsection{Examples}\label{sub:3four}
In this section, we put the Weierstrass representation (Theorem~\ref{thm:wRep}) to test, and construct explicit examples of discrete cmc surfaces in isotropic space, obtaining closed-form parametrizations.

\textbf{Doubly channel cmc-$H$ surfaces:} Let $g$ be the discrete holomorphic function that is a scaling of the identity function (with cross-ratios $-1$), given by
    \[
        g(m,n) = H \left(\frac{m}{M} + \mathbbm{i} \frac{n}{M}\right)
    \]
for some $H \in \mathbb{R}^\times$ and $M \in \mathbb{N}$.
Since the cross-ratios factorizing function of $g$ is only determined up to a multiplication by a real constant, let us choose
    \[
        m_{ij} = -m_{i\ell} = \frac{M N}{H}
    \]
for some $N \in \mathbb{N}$, and define $h$ by
    \[
        h(m,n) = H \left(\frac{m}{N} + \mathbbm{i} \frac{n}{N}\right),
    \]
Then we can check that so defined $g$ and $h$ satisfies \eqref{eqn:ghcond}, that is,
    \[
        \dif{h}_{ij} = \frac{H}{m_{ij} \dif{g}_{ij}}.
    \]

Thus, we can use Weierstrass representation so that \eqref{eqn:cmcWrep2} implies
    \begin{align*}
        \dif{Y}_{ij} &= \frac{1}{H}\Re\left((\bar{h}_{ij} + g_{ij}, 1, -\mathbbm{i}) \dif{h}_{ij}\right) = -\left(\frac{H(M+N)(2m + 1)}{2 M N^2}, \frac{1}{N},0 \right) \\
        \dif{Y}_{i\ell} &= \frac{1}{H}\Re\left((\bar{h}_{i\ell} + g_{i\ell}, 1, -\mathbbm{i}) \dif{h}_{i\ell}\right) = -\left(\frac{H(M-N)(2n + 1)}{2 M N^2}, 0, \frac{1}{N}\right).
    \end{align*}
From these recurrence relations, we can find directly that
    \[
        Y(m,n) = \left(\frac{H}{2M}\left((M+N)\Big(\frac{m}{N}\Big)^2 + (M-N)\Big(\frac{n}{N}\Big)^2\right), \frac{m}{N}, \frac{n}{N} \right).
    \]
The explicit discrete parametrization tells us that every $m$-curvature line and $n$-curvature line lies on a Euclidean parabola with vertical axis, that is, every curvature line is a discrete parabolic circle.
Therefore, the surface is a doubly channel cmc-$H$ surface (see Figure~\ref{fig:channel}).

\begin{figure}
    \centering
    \begin{minipage}{0.32\linewidth}
        \centering
        \includegraphics[width=\linewidth]{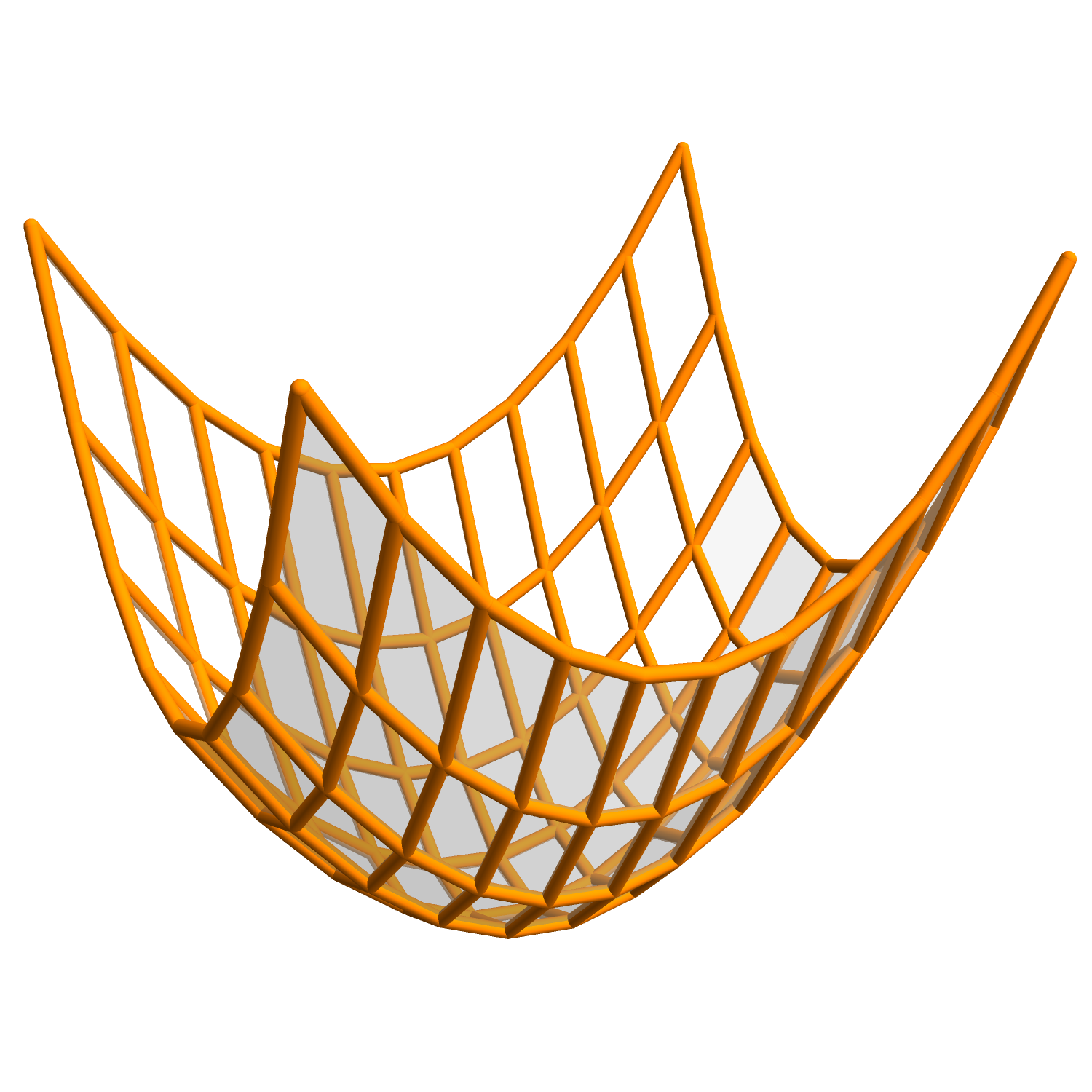}
    \end{minipage}
    \begin{minipage}{0.32\linewidth}
        \centering
        \includegraphics[width=\linewidth]{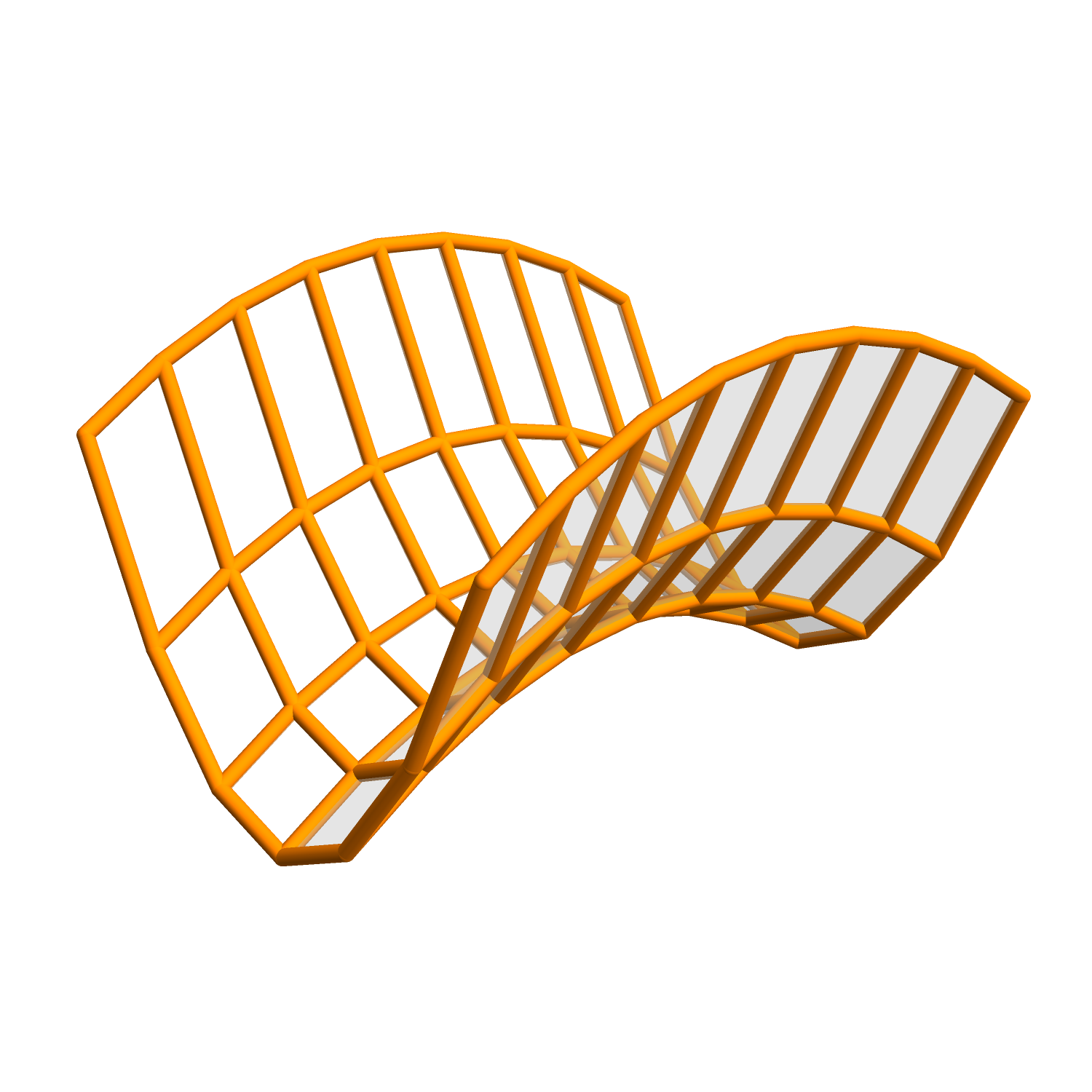}
    \end{minipage}
    \hfill
    \begin{minipage}{0.32\linewidth}
        \centering
        \includegraphics[width=\linewidth]{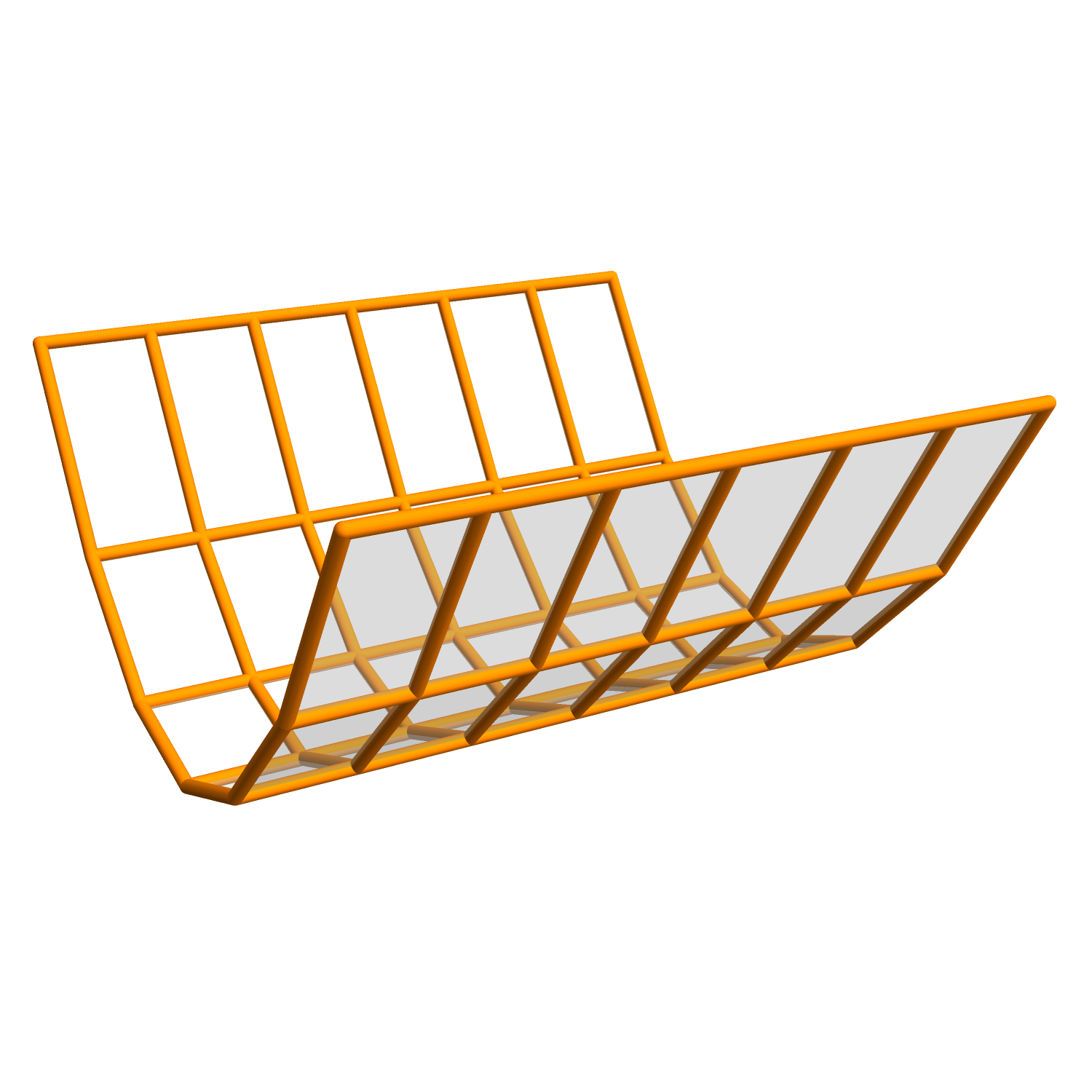}
    \end{minipage}
    
    \vspace{-25pt}
    
    \begin{minipage}{0.64\linewidth}
        \centering
        \subcaptionbox{Doubly channel cmc surfaces \label{fig:channel}}{\phantom{(A) Doubly channel cmc surface}}
    \end{minipage}
    \hfill
    \begin{minipage}{0.32\linewidth}
        \centering
        \subcaptionbox{Discrete cmc cylinder \label{fig:cylinder}}{\phantom{(B) Discrete cmc cylinder}}
    \end{minipage}
    \caption{Examples of discrete cmc-$1$ surfaces with circular curvature lines. The left-hand figure is created with $(M,N) = (6,2)$, the middle figure with $(M,N) = (-2, 3)$, and the right-hand figure with $(M, N) = (4,4)$.}
    \label{fig:channelFam}
\end{figure}

\textbf{Discrete cmc-$H$ cylinders:} From the doubly channel example, let us take $M=N$.
Then the parametrization of the discrete cmc-$H$ surface reads
    \[
        Y(m,n) = \left(H\Big(\frac{m}{M}\Big)^2, \frac{m}{M}, \frac{n}{M} \right).
    \]
Thus, this is the special case among double channel examples where the $n$-curvature lines are straight lines, giving us discrete cmc-$H$ cylinders (see Figure~\ref{fig:cylinder}).

\textbf{Discrete Delaunay-type surfaces:} Let $g^c$ be the family of discrete exponential functions (see, for example, \cite[\S~2.2]{hoffmann_DiscreteFlatSurfaces_2012} or \cite[\S~3.2]{muller_PlanarDiscreteIsothermic_2016}) with cross-ratios $-1$, so that
    \[
        g^c(m,n) = c \cdot \exp{\left(-\cosh^{-1}(2 - \cos \tfrac{\pi}{N})m + \tfrac{\mathbbm{i}\pi}{N}n\right)} =: c \cdot \exp(-\alpha m + \mathbbm{i}\beta n)
    \]
for some $N \in \mathbb{N}$ and $c \in \mathbb{R}^\times$.
Note that $g^c$ is $2N$-periodic in the $n$-direction, that is,
    \[
        g^c(m,n) = g^c(m,n+2N).
    \]
As the cross-ratios factorizing function of $g$ is only determined up to a multiplication by a real constant, let us choose
    \[
        m_{ij} = - m_{i\ell} = -\frac{1}{4c}\csc^2 \left(\frac{\pi}{2N}\right).
    \]

Under this setting, if we define $h$ via
    \[
        h(m,n) = \frac{H}{c} \, g^c(-m,-n),
    \]
then $h$ is independent of the constant $c$; furthermore, one can directly verify that $h$ satisfies \eqref{eqn:ghcond}
    \[
        \dif{h}_{ij} = \frac{H}{m_{ij} \dif{g}^c_{ij}}
    \]
on any edge $(ij)$.

Thus, we may use Weierstrass representation and evaluate \eqref{eqn:cmcWrep2} to find that
    \begin{align*}
        \dif{Y}^c_{ij} &= \begin{multlined}[t]
            \Big( -\sinh \alpha (H e^{\alpha(2m + 1)} + c),(e^{\alpha m} - e^{\alpha (m+1)}) \cos{\beta n},\\
                -(e^{\alpha m} - e^{\alpha (m+1)}) \sin{\beta n} \Big)
        \end{multlined} \\
        \dif{Y}^c_{i\ell} &= \left(0, e^{\alpha m} \left(\cos{\beta n} - \cos{\beta(n+1)}\right), -e^{\alpha m} \left(\sin{\beta n} - \sin{\beta(n+1)}\right)\right).
    \end{align*}
This allows us to recover the explicit parametrization of the surface $Y^c$ as
    \[
        Y^c(m,n)=\left( \tfrac{H}{2} e^{2 \alpha m} + c \,(\sinh{\alpha}) \, m,e^{\alpha m} \cos \tfrac{\pi n}{N}, -e^{\alpha m} \sin \tfrac{\pi n}{N} \right).
    \]
The parametrization tells us that $Y^c$ gives a $1$-parameter family of discrete surfaces of revolution with period $2N$.
These are the discrete Delaunay-type surfaces in isotropic space (see Figure~\ref{fig:delaunayFam}).

\begin{figure}
    \centering
    \begin{minipage}{0.3\linewidth}
        \centering
        \includegraphics[width=\linewidth]{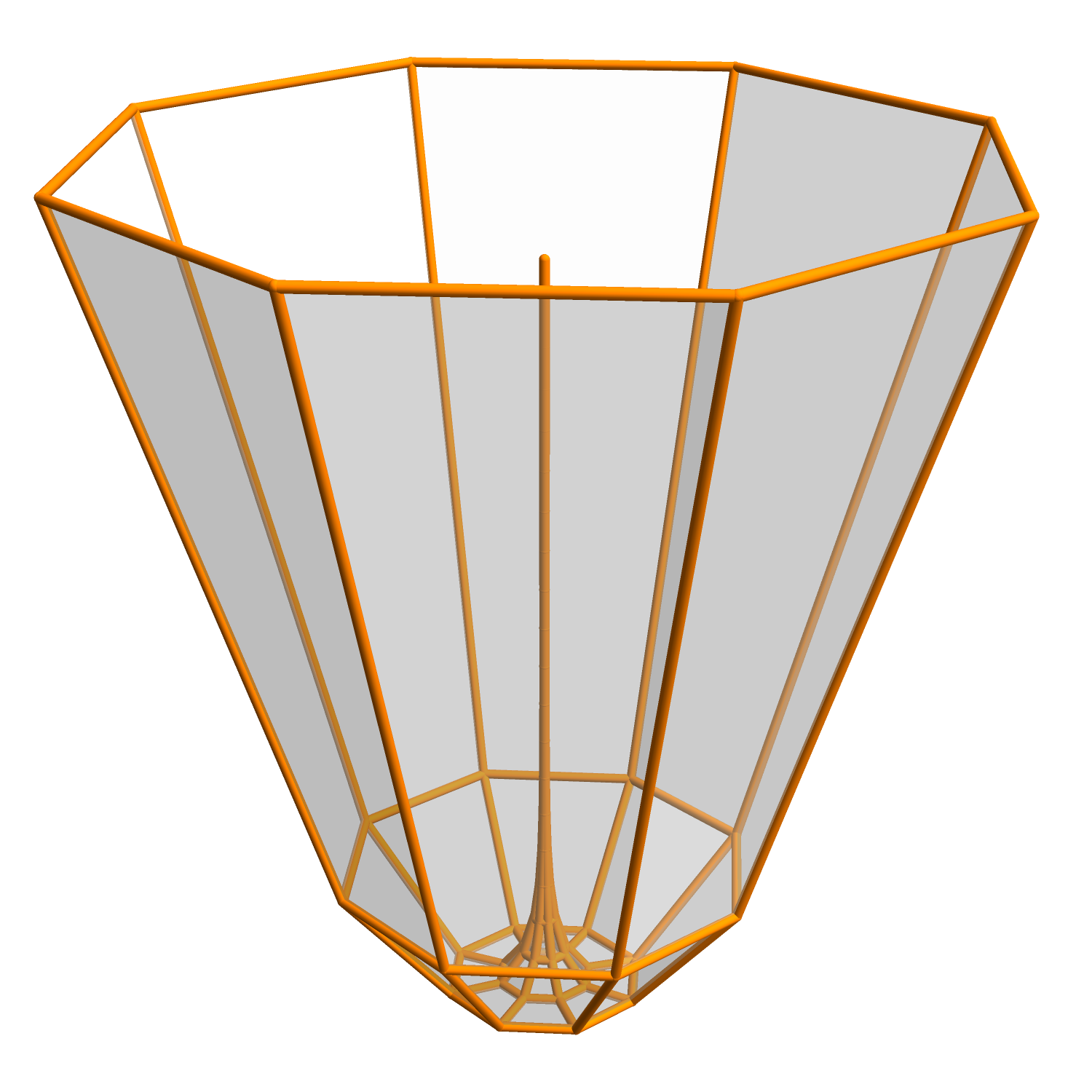}
    \end{minipage}
    \hfill
    \begin{minipage}{0.34\linewidth}
        \centering
        \includegraphics[trim={1cm 0 1cm 0}, clip, width=\linewidth]{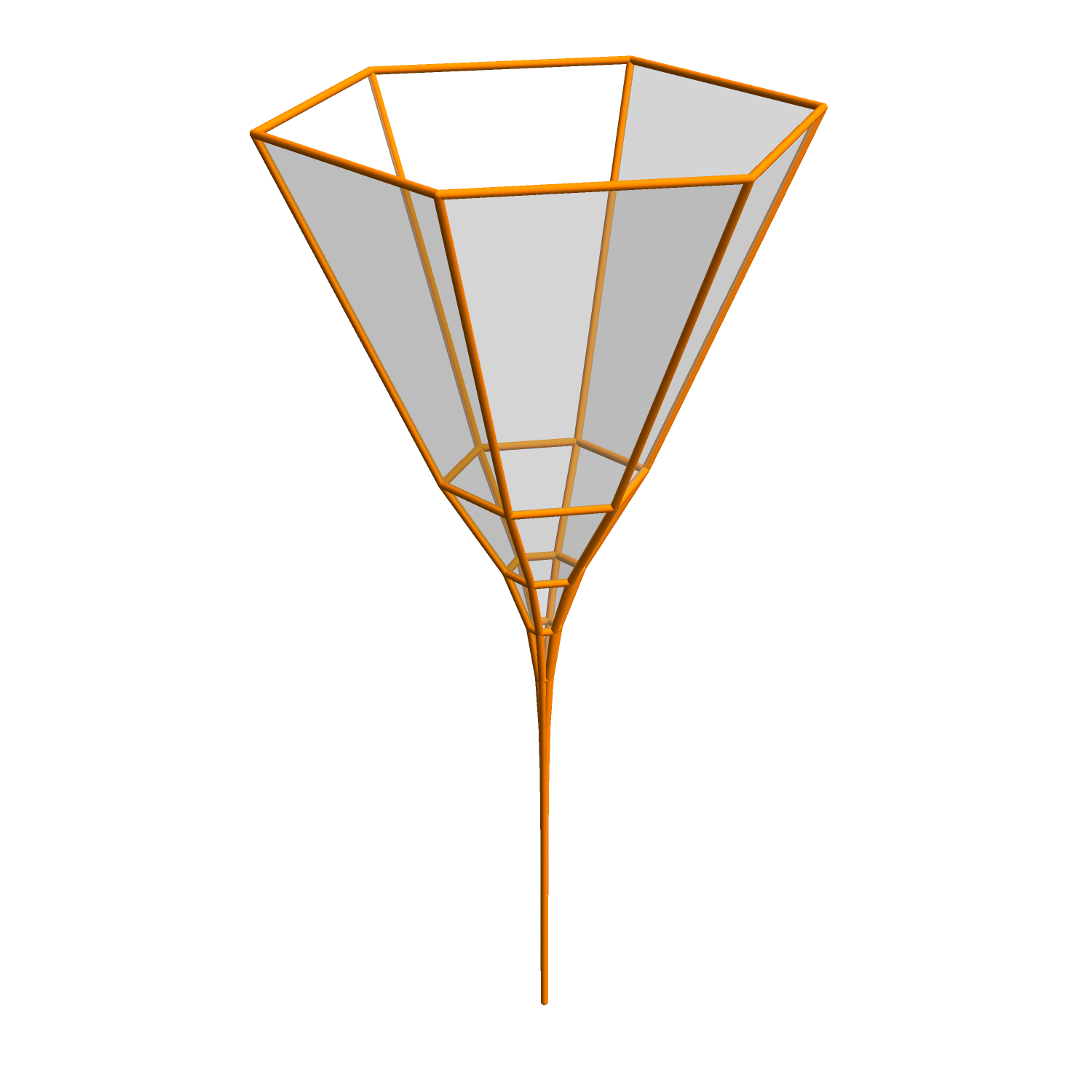}
    \end{minipage}
    \begin{minipage}{0.34\linewidth}
        \centering
        \includegraphics[trim={1cm 0 1cm 0}, clip, width=\linewidth]{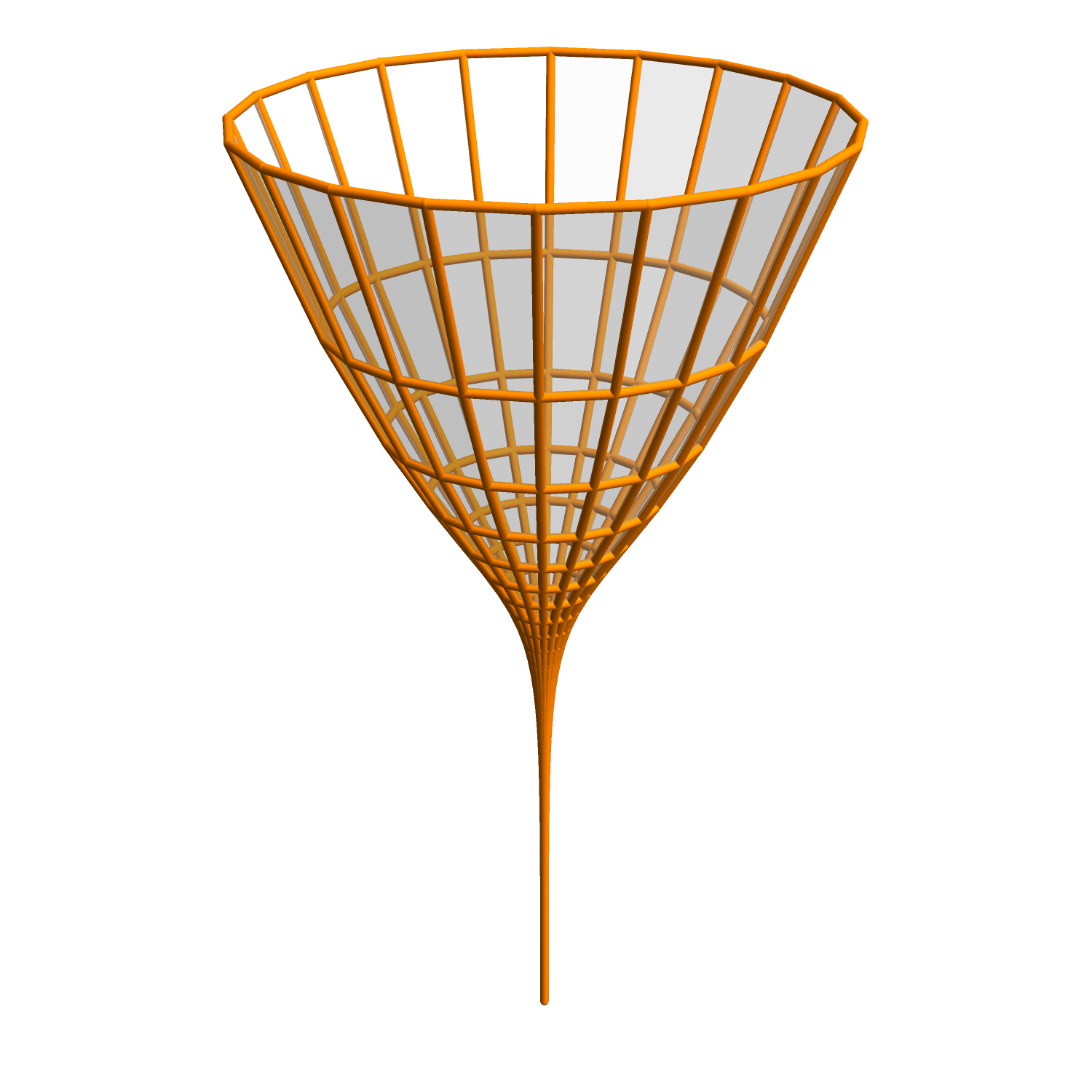}
    \end{minipage}
    \caption{Discrete Delaunay-type surfaces with $H \equiv 1$. The left-hand figure is created with $(N,c) = (4, -\tfrac{1}{2})$, the middle figure with $(N,c) = (3, \tfrac{1}{2})$, and the right-hand figure with $(N,c) = (10, \tfrac{1}{2})$.}
    \label{fig:delaunayFam}
\end{figure}
%
%



\end{document}